\documentclass[12pt,twoside, reqno]{amsart} 

\usepackage[margin=3cm]{geometry}

\newif\ifpdf
\pdffalse
\ifx\pdfoutput\undefined
\else
  \ifx\pdfoutput\relax
    \let\pdfoutput\undefined
  \else
    \ifcase\pdfoutput
    \else
      \pdftrue
    \fi
  \fi
\fi

\ifpdf
  \usepackage[pdftex]{graphicx}
\else
  \usepackage[dvips]{graphicx}
\fi

\ifpdf
 \DeclareGraphicsExtensions{.pdf, .png, .jpg, .jpeg, .tif, .tiff}
\else
 \DeclareGraphicsExtensions{.eps,.ps}
\fi

\usepackage[T1]{fontenc} 
\usepackage[latin1]{inputenc} 
\usepackage{amsmath} 
\usepackage{amsthm} 
\usepackage{amsfonts} 
\usepackage{amssymb}
\usepackage{bbm} 
\usepackage[english]{babel} 
\usepackage{varioref} 
\usepackage{enumerate} 

\usepackage{ae}

\theoremstyle{definition}
\newtheorem{problem}{Problem}

\theoremstyle{remark}

\theoremstyle{plain}
\newtheorem{thm}{Theorem}
\newtheorem*{thm*}{Theorem}
\newtheorem{lem}[thm]{Lemma}
\newtheorem{prop}[thm]{Proposition}
\newtheorem{cor}[thm]{Corollary}

\newcommand{\norm}[1]{\ensuremath{\left\Vert #1 \right\Vert}}
\newcommand{\abs}[1]{\ensuremath{\left\vert #1 \right\vert}}

\DeclareMathOperator{\dist}{dist}

\DeclareMathOperator{\spn}{span}

\newcommand{\ux}{\ensuremath{\underline{x}}}
\newcommand{\uy}{\ensuremath{\underline{y}}}

\title{Diophantine exponents for mildly restricted approximation}
\author{Yann Bugeaud} 

\address{Y. Bugeaud, Université Louis Pasteur, Mathématiques, 7, rue
  René Descartes, F-67084 Strasbourg cedex, France} 

\author{Simon Kristensen} 

\address{S. Kristensen, Department of Mathematical Sciences, Faculty
  of Science, University of Aarhus, Ny Munkegade, Building 530,
  DK-8000 Aarhus C, Denmark}

\thanks{The present research is supported by a collaborative grant
  from the French Embassy in Denmark. SK is a Steno Research Fellow
  funded by the Danish Natural Science Research Council.}

\begin{document}

\begin{abstract}
  We are studying the Diophantine exponent $\mu_{n,\ell}$ defined for
  integers $1 \leq \ell < n$ and a vector $\alpha \in \mathbb{R}^n$ by
  letting $\mu_{n,\ell} = \sup\{\mu \geq 0: 0 < \Vert \ux \cdot
  \alpha\Vert < H(\ux)^{-\mu} \allowbreak \text{for infinitely many } \ux \in
  \mathcal{C}_{n,\ell} \cap \mathbb{Z}^n \}$, where $ \cdot $ is the
  scalar product and \mbox{$\Vert \cdot \Vert$} denotes the distance
  to the nearest integer and $\mathcal{C}_{n,\ell}$ is the generalised
  cone consisting of all vectors with the height attained among the
  first $\ell$ coordinates. We show that the exponent takes all values
  in the interval $[\ell+1, \infty)$, with the value $n$ attained for
  almost all $\alpha$. We calculate the Hausdorff dimension of the set
  of vectors $\alpha$ with $\mu_{n,\ell} (\alpha) = \mu$ for $\mu \geq
  n$. Finally, letting $w_n$ denote the exponent obtained by removing
  the restrictions on $\ux$, we show that there are vectors $\alpha$
  for which the gaps in the increasing sequence $\mu_{n,1} (\alpha)
  \leq \cdots \leq \mu_{n,n-1} (\alpha) \leq w_n (\alpha)$ can be
  chosen to be arbitrary.
\end{abstract}

\maketitle

\section{Introduction}
\label{sec:introduction}

Throughout the present paper, $\Vert \cdot \Vert$ denotes the distance
to the nearest integer.  For a $n$-tuple $\alpha = (\alpha_1, \ldots ,
\alpha_n)$ of real numbers, let denote by $w_n(\alpha)$ the supremum
of the real numbers $w$ such that the inequality
\begin{equation*}
  0 < ||x_1 \alpha_1 + \ldots + x_n \alpha_n || \le H(\ux)^{-w},  
\end{equation*}
has infinitely many solutions in integer $n$-tuples $\ux = (x_1,
\ldots , x_n)$ of height $H(\ux)$, where $H(\ux) = \max\{|x_1|, \ldots
, |x_n|\}$.  This is the most classical exponent of Diophantine
approximation.  Further exponents have been introduced recently by
Bugeaud and Laurent \cite{MR2266457}.

Approximation problems closely related to the study of the exponents
$w_n$ were considered by Jarn\'ik
\cite{jarnik37:_ueber_loesun_gleic_zahlen}, Schmidt \cite{MR0429762}
and Thurnheer \cite{MR649949, MR672846, MR774100, MR777764,
  MR1056107}. In these papers relatively mild restrictions are placed
on the integer vectors $\ux$. In Jarn\'ik's paper
\cite{jarnik37:_ueber_loesun_gleic_zahlen}, the additional restriction
was put on $\ux$ that at least $\ell$ of its coordinates had to be
non-zero. In the papers by Schmidt and Thurnheer, stronger
restrictions were made, all of which can be viewed as special cases of
the ones considered in the present paper, where we restrict the $\ux$
to a rectangular cone (see below).

We introduce and study the following exponents of restricted
approximation.  Let $1 \le \ell < n$ be integers and $\alpha =
(\alpha_1, \ldots , \alpha_n)$ be a real $n$-tuple.  We denote by
$\mu_{n, \ell}(\alpha)$ the supremum of the real numbers $\mu$ such
that the inequality
\begin{equation*}
  0 < ||x_1 \alpha_1 + \ldots + x_n \alpha_n || \le H(\ux)^{-\mu}
\end{equation*}
has infinitely many solutions in integer $n$-tuples $\ux = (x_1,
\ldots , x_n)$ satisfying
\begin{equation}
  \label{eq:4}
  \max\{|x_{\ell+1}|, \ldots , |x_n|\} < \max\{|x_1|, \ldots ,
  |x_{\ell}|\}. 
\end{equation}

This simply means that we impose that the height of $\ux$ is attained
among its $\ell$ first coordinates. We write $\mu_{n, \ell}$ (resp.
$w_n$) instead of $\mu_{n, \ell}(\alpha)$ (resp. $w_n (\alpha)$) when
there is no confusion.  By extension and for consistency of notation,
we define $\mu_{n,n} = w_n$.

Since the $\alpha_i$ do not play the same role, the situation is not
symmetrical, thus difficulties of a new kind occur.  It is different
from inhomogeneous approximation, since we have here less constraints.
Geometrically, we are restricting the `denominators' $\ux$ to lie in
some rectangular cone. In Schmidt's original paper \cite{MR0429762},
$n = 2$ and the $\ux$ were restricted to the first quadrant. The
rotated setting is better suited to our purposes, and it causes no
loss of generality as was also remarked by Schmidt.  Actually, our
results remain valid if \eqref{eq:4} is replaced by
\begin{equation*}
  \max\{|x_{\ell+1}|, \ldots , |x_n|\} < C \, \max\{|x_1|, \ldots ,
  |x_{\ell}|\},
\end{equation*}
where $C$ is an arbitrary given positive number. 

The exponent $\mu$ defined by Schmidt \cite{MR0429762} is simply
$\mu_{2,1}$ with our notation. The exponents $\mu_{n, n-1}$ correspond
to those introduced by Thurnheer \cite{MR1056107}. One of our aims in
the present psper is to show that, from a metric point of view, all
the exponents $\mu_{n, \ell}$ with $\ell = 1, \ldots , n$ have a
similar behaviour (Theorem 4).  In the opposite direction, we
construct explicit examples of $n$-tuples $\alpha$ for which all the
$\mu_{n, \ell} (\alpha)$, $1 \le \ell \le n$, are different (Theorem
2). We further investigate (Theorem 1) the set of values taken by the
functions $\mu_{n, \ell}$.

\section{Results}
\label{sec:results}

In the present paper, we are mainly concern with the {\it spectra} of
the exponents of Diophantine approximation, that is, with the set of
values taken by $\mu_{n, \ell}$ on the set of real $n$-tuples whose
coordinates are, together with $1$, linearly independent over the
rationals. The reason for the latter restriction on the set of
$n$-tuples is to avoid pathologies within the setup. Indeed, if we did
have linear dependence, we would essentially be studying a lower
dimensional problem, and the resulting spectrum would incorporate such
lower dimensional phenomena. This would in turn obscure the nature of
the exponent.

For convenience,
unless the contrary is stated explicitly, we assume that the
coordinates of the real $n$-tuples occurring from now on are, together
with $1$, linearly independent over the rationals.

Choosing $x_{\ell+1} = \cdots = x_n = 0$ and applying Dirichlet's {\it
  Schubfachprinzip}, we easily get that $\mu_{n, \ell} \ge \ell$, for
any positive integers $\ell$ and $n$ with $1 \le \ell < n$.
Furthermore, since there are $n$ free coefficients, namely $x_1,
\ldots , x_n$, in the definition of $\mu_{n, \ell}$, we can reasonably
expect $\mu_{n, \ell}$ to be often at least equal to $n$.  However, as
noted by Schmidt \cite{MR0429762}, for any positive $\epsilon$, there
exist real $n$-tuples $\alpha$ with $\mu_{n, 1}(\alpha) \le 2 +
\epsilon$.

\begin{thm}
  \label{thm:main1}
  Let $\ell$ and $n$ be positive integers with $1 \le \ell \leq n$.
  Then, $\mu_{n, \ell} (\alpha) = n$ for almost all real $n$-tuples
  $\alpha$.  Furthermore, for any real number $\mu_{\ell}$ with
  $\mu_{\ell} \ge \ell + 1$, there exist uncountably many $n$-tuples
  $\alpha$ having $\mu_{n, \ell} (\alpha) = \mu_{\ell}$.
\end{thm}

In fact, we prove a more precise result than the first assertion (see
Theorem \ref{thm:lebesgue} and Theorem \ref{thm:hausdorff} in Section
\ref{sec:metrical-theory}).  Namely, we establish a zero--one law for
Lebesgue measure and a zero-infinity law for Hausdorff measure, which
implies that from the metrical point of view, all the exponents
$\mu_{n,1}, \dots, \mu_{n,n-1}$ and $w_n = \mu_{n,n}$ have the same
behaviour. However, in view of the remarks preceeding the statement of
Theorem \ref{thm:main1} together with the theorem, there is a distinct
difference in the spectrum of these exponents.  In particular, the
fact (established in Section \ref{sec:small-expon}) that the spectrum
of $\mu_{n, \ell}$ includes the interval $[\ell + 1, n]$ is the most
interesting part of Theorem \ref{thm:main1}.
 
For $(n, \ell) = (2, 1)$, the first assertion of Theorem
\ref{thm:main1} was proved by Thurnheer \cite{MR777764}.  The main
tool for the proof of Theorem \ref{thm:main1} is the theory of
Hausdorff measure.  Consequently, it does not yield to {\it explicit}
examples of $n$-tuples $\alpha$ with prescribed values for $\mu_{n,
  \ell} (\alpha)$.  However, inspired by Schmidt's construction of
$T$-numbers \cite{MR0279043, MR0344204}, we give an effective
construction of $n$-tuples $\alpha$ with prescribed exponents $\mu_{n,
  \ell} (\alpha)$, provided that these values are sufficiently large.
This approach enables us also to prove that the difference between
$w_n (\alpha)$ and $\mu_{n, \ell} (\alpha)$ can be arbitrarily large.
In the statement of the next theorem, established in Section
\ref{sec:diff-betw-mu_n}, we adopt the convention that $+ \infty + x =
+ \infty$ for any real number $x$.

\begin{thm}
  \label{thm:main2}
  Let $n \ge 2$ be an integer.  Let $\delta_1, \ldots , \delta_{n-1}$
  be elements of $\mathbb{R}_{\ge 0} \cup \{+\infty\}$.  Then, there
  exist uncountably many $n$-tuples $\alpha$ having $w_n(\alpha) =
  \mu_{n, n-1} (\alpha) + \delta_{n-1}$ and $\mu_{n, \ell} (\alpha) =
  \mu_{n, \ell - 1} (\alpha) + \delta_{\ell -1}$, for $\ell = 2,
  \ldots , n-1$.
\end{thm}

The proof of Theorem \ref{thm:main2} rests on the effective
construction of real numbers $\xi$ for which the $n$-tuples $(\xi^n,
\ldots, \xi)$ have the required properties.

The above theorems say nothing on the values of the spectrum of
$\mu_{n, \ell}$ belonging to the interval $[\ell, \ell + 1)$.  Schmidt
proved that $\mu_{2, 1} (\alpha) \ge ( 1 + \sqrt{5})/2$, a result
extended to the exponents $\mu_{n, n-1}$ by Thurnheer \cite{MR1056107}. 
In particular, the following problems remain open.

\begin{problem}
  Let $\ell$ and $n$ be integers
  with $1 \le \ell < n$.  To prove or to disprove that there exist
  $\alpha$ such that
  \begin{equation*}
    \mu_{n, \ell} (\alpha) < \ell + 1.
  \end{equation*}
\end{problem}

For $n=2$, Problem 1 was previously posed by 
Schmidt \cite{MR0429762, MR0702204}.

\begin{problem}
  To establish a uniform lower bound for $\mu_{n, 1}$ that tends to
  $2$ when $n$ tends to infinity.
\end{problem}

We have been unable to make any progress on these questions (see
however Section \ref{sec:restr-appr-along}).  Nevertheless, for sake
of completeness, we restate the lower bounds obtained by Schmidt and
Thurnheer, by making use of the exponents of Diophantine approximation
$w_n$ and $\hat{w}_n$, the latter being defined as follows.  For an
integer $n \ge 1$ and a real $n$-tuple $\alpha = (\alpha_1, \ldots ,
\alpha_n)$, we denote by $\hat{w}_n(\alpha)$ the supremum of the real
numbers $\hat{w}$ such that, for any real number $X > 1$, the
inequality
\begin{equation*}
  0 < ||x_1 \alpha_1 + \ldots + x_n \alpha_n || \le X^{-\hat{w}}
\end{equation*}
has an integer solution $\ux = (x_1, \ldots , x_n)$ satisfying $H(\ux)
\le X$.

\begin{thm}
  \label{thm:main3}
  Let $\ell$ and $n$ be integers with $1\le \ell < n$.  For any real
  $n$-tuple $\alpha$, we have
  \begin{equation}
    \label{eq:5}
    \mu_{n, \ell} (\alpha) \ge  \frac{\ell \, \hat{w}_n (\alpha)}
    {\hat{w}_n (\alpha) - n + \ell}
  \end{equation}
  and
  \begin{equation*}
  \mu_{n, n-1} (\alpha) \ge \hat{w}_n (\alpha) - 1 + 
  \frac{\hat{w}_n (\alpha)}{w_n (\alpha)}.
\end{equation*}
\end{thm}

The second inequality from Theorem \ref{thm:main3} implies that
$\mu_{n, n-1} \ge \hat{w}_n - 1$. Combined with \eqref{eq:5}, this
yields the lower bound
\begin{equation}
  \label{eq:Thurn}
  \mu_{n, n-1} (\alpha) \ge \frac{n - 1 + \sqrt{n^2 + 2n - 3}}{2},
\end{equation}
for any $\alpha$ in $\mathbb{R}^n$.  This was established in 1976 by
Schmidt \cite{MR0429762} for $n=2$ and in 1990 by Thurnheer
\cite{MR1056107} for arbitrary $n$.  Observe that the right hand side
of \eqref{eq:Thurn} is greater than $n - 1/n$.

Furthermore, it should be noted that Schmidt's and Thurnheer's results
are slightly more precise, since they assert the existence of a
positive constant $C$ such that, for any real $n$-tuple $\alpha$
(whose coordinates, together with $1$, are linearly independent over
the rationals) there exist infinitely many integer $n$-tuples $\ux =
(x_1, \ldots , x_n)$ satisfying
\begin{equation*}  
  0 < ||x_1 \alpha_1 + \ldots + x_n \alpha_n || \le C
  H(\ux)^{-(n - 1 + \sqrt{n^2 + 2n - 3})/2}
\end{equation*}
and
\begin{equation*}
|x_n| < \max\{|x_1|, \ldots , |x_{n-1}|\}. 
\end{equation*}

Although Theorem \ref{thm:main3} is essentially proved in the papers
by Schmidt and Thurnheer, we include a proof of \eqref{eq:5}, 
postponed to Section \ref{sec:lower-bounds-expon}.

Throughout, we use the Vinogradov notation and write $a \ll b$ if
there is a constant $C > 0$ such that $a \leq Cb$. If $a \ll b$ and $b
\ll a$, we write $a \asymp b$. Furthermore,
$\dim(E)$ denotes the Hausdorff dimension of the set $E$.

\section{The metrical theory for the exponents $\mu_{n,\ell}$}
\label{sec:metrical-theory}

It is the purpose of the present section to show that the exponent
$\mu_{n, \ell}$ takes all values in the interval $[n, +\infty)$. This
follows from a more general metrical result, which gives a
complete metrical description of the sets
\begin{alignat*}{2}
  \mathcal{L}_{n,\ell}(\psi) = \big\{(\alpha_1, &\dots, \alpha_n) \in
  \mathbb{R}^n : \norm{\alpha_1 x_1 + \dots + \alpha_n x_n} \leq
  \psi(H(x))\\ &\text{ for infinitely many } \ux = (x_1, \dots, x_n) \in
  \mathbb{Z}^n \\ &\text{ with } \max\{\abs{x_{\ell+1}}, \dots,
  \abs{x_n}\} < \max\{\abs{x_1}, \dots, \abs{x_{\ell}}\}\big\}
\end{alignat*}
Our result also includes the case $\ell = n$, where the last condition
on the integer vectors $\ux$ is empty and Theorem \ref{thm:lebesgue}
below reduces to a classical result of Groshev \cite{groshev38}.

\begin{thm}
  \label{thm:lebesgue}
  Let $\psi : \mathbb{Z}_{\geq 0} \rightarrow \mathbb{R}_{>0}$ be a
  non-increasing function, let $n,\ell$ be integers with $1 \leq
  \ell \leq n$.  Then, $\mathcal{L}_{n,\ell}$ is null (resp.~full)
  according to the convergence (resp.~divergence) of the series
  \begin{equation*}
    \sum_{n=1}^\infty h^{n-1}\psi(h).
  \end{equation*}
\end{thm}

\begin{proof}[Proof of Theorem \ref{thm:lebesgue} (convergence part)]
  The case of convergence is a consequence of the usual
  Khintchine--Groshev theorem (see, \emph{e.g.}, \cite{MR1727177}),
  since $\mathcal{L}_{n,\ell}$ is a subset of the corresponding set
  without restrictions on points $\ux$. Since the measure of the
  larger set is zero in the case of convergence, the convergence half
  follows.
\end{proof}

The case of divergence will be derived from the following result,
which is the simplest version of the divergence part of the
Borel--Cantelli Lemma.
\begin{lem}\label{lem2}
  Let $(\Omega,A,\mu) $ be a probability space 
  and $(E_n)_{n \ge 1}$ be a sequence
  of $\mu$-measurable sets such that $\sum_{n=1}^\infty
  \mu(E_n)=\infty $. Suppose that whenever $m \neq n$,
  \begin{equation*}
    \mu(E_m \cap E_n) = \mu(E_m) \mu(E_n).
  \end{equation*}
  Then,
  \begin{equation*}
    \mu( \limsup_{n \to \infty} E_n ) = 1
  \end{equation*}
\end{lem}

\begin{proof}[Proof of Theorem \ref{thm:lebesgue} (divergence part)]
  In order to prove that our set has full measure, we note that
  $\mathcal{L}_{n,\ell}$ is invariant under translation by integer
  vectors. Hence it suffices to show that $\mathcal{L}_{n,\ell} \cap
  [0,1]^n$ has measure $1$.  We consider the sets
  \begin{equation*}
    B(x_1, \dots, x_n) = \left\{(\alpha_1, \dots, \alpha_n) \in
      [0,1]^n : \norm{x_1\alpha_1 + \cdots + x_n\alpha_n} \leq
      \psi(H(\ux))\right\},
  \end{equation*}
  where $\ux = (x_1, \dots, x_n) \in \mathbb{Z}^n$. It is easy to see
  that
  \begin{equation}
    \label{eq:1}
    \abs{B(x_1, \dots, x_n)} \asymp \psi(H(\ux)),
  \end{equation}
  where $\abs{B}$ denotes the Lebesgue measure of the set $B$.
  Furthermore, if $(x_1, \dots, x_n)$ and $(x'_1, \dots, x'_n)$ are
  linearly independent, then
  \begin{equation}
    \label{eq:2}
    \abs{B(x_1, \dots, x_n) \cap B(x'_1, \dots, x'_n)}=\abs{B(x_1,
      \dots, x_n)} \cdot \abs{B(x'_1, \dots, x'_n)}.
  \end{equation}
  This is proved in, \emph{e.g.}, \cite{MR1258756}.

  We will impose further restrictions on the $\ux$'s in order to
  ensure that \eqref{eq:2} holds for any pair of distinct vectors. Any
  vector $\alpha$ satisfying infinitely many of the further restricted
  inequalities automatically lies within $\mathcal{L}_{n,\ell}$.
  Hence, a lower bound on the estimate on the measure of the further
  restricted set implies a lower bound on the measure of the original
  set.

  We define
  \begin{alignat*}{2}
    P_N = \big\{(x_1, \dots, x_n) \in \mathbb{Z}_{\geq0}^n \ : \
    &H(\ux) =
    N, x_n \geq 1 \\
    &\gcd(x_1, \dots, x_n) = 1 \text{ and } \\
    & 2 \max\{\abs{x_{\ell+1}}, \dots, \abs{x_n}\} \leq
    \max\{\abs{x_1}, \dots, \abs{x_{\ell}}\}\big\}.
  \end{alignat*}
  If $\ux \in P_N$ and $\ux' \in P_{N'}$ are linearly dependent, then
  for some integer $r \in \mathbb{Z}$, $\ux = r \ux'$ or $r \ux =
  \ux'$. In either case, by assumption of coprimality, $r = \pm 1$,
  and since the last coordinates are assumed to be positive, $r =1$,
  whence $\ux = \ux'$. Hence, \eqref{eq:2} holds for any pair of
  distinct vectors $\ux \in P_N$ and $\ux' \in P_{N'}$.

  Let $\mu : \mathbb{Z}_{\geq 0} \rightarrow \{-1,0,1\}$ denote the
  M\"obius function, \emph{i.e.},
  \begin{equation*}
    \mu(n) =
    \begin{cases}
      0 & \text{if } n=0, \\
      1 & \text{if } n=1, \\
      (-1)^k & \text{if $n$ has $k$ distinct prime factors.}
    \end{cases}
  \end{equation*}
  We use the identity
  \begin{equation*}
    \sum_{d \vert n} \mu(d) =
    \begin{cases}
      1 & \text{if } n=1, \\
      0 & \text{otherwise.}
    \end{cases}
  \end{equation*}
  With this identity, we can extimate the number of elements of $P_N$
  as follows,
  \begin{alignat*}{2}
    \# P_N &= \sum_{\substack{H(\ux) = N \\ \gcd(x_1, \dots, x_n) = 1
        \\ 2 \max\{\abs{x_{\ell+1}}, \dots, \abs{x_n}\} \leq
        \max\{\abs{x_1},
        \dots, \abs{x_{\ell}}\}}} 1 \\
    &= \sum_{\substack{H(\ux) = N \\ \gcd(x_1, \dots, x_n) = k \\ 2
        \max\{\abs{x_{\ell+1}}, \dots, \abs{x_n}\} \leq
        \max\{\abs{x_1},
        \dots, \abs{x_{\ell}}\}}} \sum_{d \vert k} \mu(d) \\
    &= \sum_{d \vert N} \mu(d) \sum_{\substack{H(\ux') = N/d \\ 2
        \max\{\abs{x_{\ell+1}}, \dots, \abs{x_n}\} \leq
        \max\{\abs{x_1},
        \dots, \abs{x_{\ell}} \}}} 1 \\
    &\asymp \sum_{d \vert N} \mu(d) \left(\frac{N}{d}\right)^{n-1}
    \ell.
  \end{alignat*}

  If $n=2$, then
  \begin{equation*}
    \sum_{d \vert N} \mu(d) \left(\frac{N}{d}\right) \ell 
    = \ell \phi(N),
  \end{equation*}
  where $\phi$ denotes the Euler totient function. If $n > 2$,
  \begin{equation*}
    \sum_{d \vert N} \mu(d) \left(\frac{N}{d}\right)^{n-1} \ell 
    = \ell N^{n-1} \sum_{d \vert N} \frac{\mu(d)}{d^{n-1}}.
  \end{equation*}
  In order to show that this is comparable to $N^{n-1}$, it suffices
  to note that 
  \begin{multline}
    \label{eq:21}
    \frac{6}{\pi^2} = \frac{1}{\zeta(2)} \leq \frac{1}{\zeta(n-1)} =
    \prod_{\text{all primes $p$}} \left(1-\frac{1}{p^{n-1}}\right) \\ <
    \prod_{\substack{p \vert N \\ \text{$p$ is prime}}}
    \left(1-\frac{1}{p^{n-1}}\right) = \sum_{d \vert N}
    \frac{\mu(d)}{d^{n-1}} < 1.
  \end{multline} 
  where $\zeta$ denotes the Riemann $\zeta$-function. Note that we
  have used the Euler product formula for this function in order to
  make the argument completely clear.

  In order to prove Theorem \ref{thm:lebesgue}, it suffices to prove
  that if the series $\sum h^{n-1}\psi(h)$ is divergent, then
  \begin{equation}
    \label{eq:3}
    \sum_{N=1}^\infty \sum_{x \in P_N} \psi(N) = \infty.
  \end{equation}
  In view of the above, this is
  immediate when $n \geq 3$. When $n=2$, we use the
  identity
  \begin{equation*}
    \sum_{r = 1}^N \phi(r) = \frac{3}{\pi^2}N^2 + O(N \log N),
  \end{equation*}
  from elementary number theory.  We split the sum \eqref{eq:3} into
  dyadic blocks to get
  \begin{alignat*}{2}
    \sum_{N=1}^\infty \sum_{x \in P_N} \psi(N) &= \sum_{k=0}^\infty
    \psi(2^{k+1}) \sum_{2^k \leq r < 2^{k+1}} \phi(r)\\
    &= \sum_{k=0}^\infty \psi(2^{k+1}) \left(\frac{9}{2 \pi^2}
      2^{2(k+1)} + O(k 2^k)\right) = \infty.
  \end{alignat*}
  The final equality follows by condensation and assumption of
  divergence. Hence, Lemma \ref{lem2} applies, and the theorem
  follows.
\end{proof}

We now turn our attention to Hausdorff measures. We have the following
theorem.
\begin{thm}
  \label{thm:hausdorff}
  Let $\psi : \mathbb{Z}_{\geq 0} \rightarrow \mathbb{R}_{>0}$ be
  non-increasing, let $f: \mathbb{R} \rightarrow \mathbb{R}$ be a
  dimension function such that $r \mapsto r^{-n}f(r)$ is monotonically
  increasing and such that $g : r \mapsto r^{-(n-1)}f(r)$ is also a
  dimension function. Then,
  \begin{equation*}
    \mathcal{H}^f (\mathcal{L}_{n,\ell}(\psi)) =
    \begin{cases}
      0 & \text{whenever } \sum_{r=1}^\infty
      r^{n}g\left(\frac{\psi(r)}{r}\right) < \infty, \\
      \infty & \text{whenever } \sum_{r=1}^\infty
      r^{n}g\left(\frac{\psi(r)}{r}\right) = \infty.
    \end{cases}
  \end{equation*}
\end{thm}

\begin{proof}
  To prove the convergence result, we cover each $B(x_1, \dots, x_n)$
  by no more than some constant times $H(\ux)^n \psi(H(\ux))^{-(n-1)}$
  balls of width $\asymp \psi(H(\ux))/H(\ux)$. Using this cover to
  bound the Hausdorff $f$-measure, we get for any $N$,
  \begin{alignat*}{2}
    \mathcal{H}^f (\mathcal{L}_{n,\ell}(\psi)) &\ll \sum_{r \geq N}
    \sum_{H(\ux) = r} f\left(\frac{\psi(r)}{r}\right) r^n
    \psi(r)^{-(n-1)}  \\
    &\ll \sum_{r \geq N} r^n \left(\frac{\psi(r)}{r}\right)^{-(n-1)}
    f\left(\frac{\psi(r)}{r}\right) \\
    &= \sum_{r \geq N} r^n g\left(\frac{\psi(r)}{r}\right) \rightarrow
    0.
  \end{alignat*}

  To get the divergence case, we apply result of Beresnevich and
  Velani \cite{MR2264714}, which combines the Hausdorff and Lebesgue
  theory for lim sup sets of the type considered here in one package.
  With reference to their setup, we let $\mathcal{R}$ be the
  collection of hyperplanes in $\mathbb{R}^n$ given by the equations
  \begin{equation*}
    R_{(x_1, \dots,x_n,y)} = \left\{\alpha \in \mathbb{R}^n : x_1
      \alpha_1 + \cdots + x_n \alpha_n = y\right\},
  \end{equation*}
  where $(x_1, \dots, x_n) \in \mathbb{Z}^n$ satisfies
  \begin{equation*}
    2\max\{\abs{x_{\ell+1}}, \dots, \abs{x_n}\} \leq \max\{\abs{x_1}, 
    \dots, \abs{x_{\ell}}\},
  \end{equation*}
  and $y \in \mathbb{Z}$. Also, let $\Upsilon(x_1, \dots, x_n, y) =
  \psi(H(\ux))/(nH(\ux))$ and let
  \begin{multline*}
    \Delta(R_{(x_1, \dots, x_n, y)}, \Upsilon(x_1, \dots, x_n, y)) \\
    = \left\{x \in \mathbb{R}^n : \dist(x, R_{x_1, \dots, x_n, y})
      \leq \Upsilon(x_1, \dots, x_n, y)\right\}.
  \end{multline*}
  It is an easy exercise to show that
  \begin{equation*}
    \limsup \Delta(R_{(x_1, \dots, x_n, y)}, \Upsilon(x_1, \dots, x_n,
    y)) \subseteq \limsup  B(x_1, \dots, x_n).
  \end{equation*}

  In order to invoke the main result of \cite{MR2264714}, we need a
  line in $\mathbb{R}^n$, such that the angle between the hyperplanes
  in $\mathcal{R}$ and the line is bounded away from zero. Note that
  the line $V = \spn\{(0, \dots, 0, 1)\}$ has this property. It now
  follows from \cite[Theorem 3]{MR2264714} together with Theorem
  \ref{thm:lebesgue} that the divergence part holds.
\end{proof}

Note that the same result holds for $\ell = n$, so the above result
contains the classical result of Jarn\'ik \cite{MR1545226} and its
extension to arbitrary Hausdorff measure in \cite{MR1468922}, where
the more general problem of systems of linear forms is considered. In
addition, the result shows that the metrical theory is indifferent to
restrictions of the form studied in this paper. As a consequence of
Theorem \ref{thm:hausdorff}, we see that the dimension result valid
for exact order sets \cite{MR1866488} in the classical case remains
valid under mild restrictions.

\begin{cor}
  \label{cor:HDim}
  Let $\mu > n$. Then,
  \begin{equation*}
    \dim\{\alpha \in \mathbb{R}^n : \mu_{n,\ell}(\alpha) = \mu\} =
    n-1+\frac{n+1}{\mu+1}.
  \end{equation*}
  In particular, the exponent $\mu_{n,\ell}$ attains all values between
  $n$ and $\infty$.
\end{cor}

This result is an exact order version of a previous result of Rynne
\cite{MR93a:11066}, who calculated the Hausdorff dimension of sets of
vectors for which the Diophantine exponent obtained by restricting the
$\ux$ to arbitrary subsets of $\mathbb{Z}^n$ is upper bounded. In our
setting, Rynne's result would give
\begin{equation*}
  \dim\{\alpha \in \mathbb{R}^n : \mu_{n,\ell}(\alpha) \leq \mu\} =
  n-1+\frac{n+1}{\mu+1}.
\end{equation*}
Clearly, the present result is stronger in the present setup, although
Rynne's result is applicable to a wider class of restrictions.

\begin{proof}
  Let $\psi(r) = r^{-\mu}$ and $\psi_0(r) = r^{-\mu}/\log^2 r$. We
  consider the set
  \begin{equation*}
    \mathcal{L}_{n,\ell}(\psi) \setminus \mathcal{L}_{n,\ell}(\psi_0).
  \end{equation*}
  This set is certainly contained in the set of the corollary. We show
  that the dimension of this set satisfies the corresponding lower
  bound. This in turn follows if we show that for $s =
  n-1+\frac{n+1}{\mu+1}$,
  \begin{equation*}
    \mathcal{H}^s (\mathcal{L}_{n,\ell}(\psi)) = \infty \text{ and }
    \mathcal{H}^s (\mathcal{L}_{n,\ell}(\psi_0)) = 0.
  \end{equation*}
  But this follows from Theorem \ref{thm:hausdorff}, since on
  inserting all definitions and reducing,
  \begin{equation*}
    \sum_{r=1}^\infty r^{n}g\left(\frac{\psi(r)}{r}\right) =
    \sum_{r=1}^\infty \frac{1}{r} = \infty,
  \end{equation*}
  whereas
  \begin{equation*}
    \sum_{r=1}^\infty r^{n}g\left(\frac{\psi_0(r)}{r}\right) =
    \sum_{r=1}^\infty \frac{1}{r \log^2 r} < \infty.
  \end{equation*}
  This completes the proof of Corollary \ref{cor:HDim}.
\end{proof}

\section{Small values of the exponents $\mu_{n, \ell}$}
\label{sec:small-expon}

As noted just above Theorem \ref{thm:main1},
Schmidt \cite{MR0429762} proved that for any positive
$\epsilon$ and any integer $n \ge 2$
there are $n$-tuples $\alpha$ such that $\mu_{n, 1}
(\alpha) \le 2 + \epsilon$. His proof can be easily modified to
assert the existence of $\alpha$ with
\begin{equation}
  \label{eq:6}
  \mu_{n, 1} (\alpha) \le 2.
\end{equation}

The purpose of the present section is to prove something more, namely
the following theorem.

\begin{thm}
  \label{thm:small_exponents}
  Let $1 \leq \ell < n$ and let $\ell + 1 \leq \mu \leq n$. Then there
  are continuum many $\alpha = (\alpha_1, \dots, \alpha_n) \in
  \mathbb{R}^n$ with $1, \alpha_1, \dots, \alpha_n$ linearly
  independent over $\mathbb{Q}$ such that
  \begin{equation*}
    \mu_{n, \ell}(\alpha) = \mu.
  \end{equation*}
\end{thm}

The proof is an extension of the method employed by Schmidt to prove
\eqref{eq:6} in \cite{MR0429762}. We need a lemma from the metrical
theory of Diophantine approximations. We first define an auxiliary
Diophantine exponent.  Let $1 \leq \ell < n$ and let
$(\alpha_{\ell+1}, \dots, \alpha_n) \in \mathbb{R}^{n-\ell}$ be fixed.
We define
\begin{align*}
  \tilde{\nu}_{n, \ell}(\alpha_{1}, \dots, \alpha_\ell) = \sup
  \bigg\{\nu > 0 : &\min_{1 \leq i \leq n - \ell} \norm{x_1 \alpha_1 +
    \cdots + x_\ell
    \alpha_\ell + x_{\ell+i} \alpha_{\ell+i}} < H(\ux)^{-\nu}\\
  &\text{ for infinitely many } \ux \in \mathbb{Z}^n \\
  &\text{ with } \max\{\abs{x_1}, \dots, \abs{x_\ell}\} >
  \max\{\abs{x_{\ell+1}}, \dots, \abs{x_n}\} \bigg\}.
\end{align*}
We use a metrical result for this exponent.

\begin{lem}
  \label{lem:auxiliary_exps}
  Let $1 \leq \ell < n$ and let $(\alpha_{\ell+1}, \dots, \alpha_n)
  \in \mathbb{R}^{n-\ell}$ be fixed. Then, for $\nu \geq \ell + 1$,
  \begin{equation*}
    \dim \left\{(\alpha_{1}, \dots, \alpha_\ell) \in
      \mathbb{R}^{\ell} : \tilde{\nu}_{n, \ell}(\alpha_{1}, \dots,
      \alpha_\ell) =\nu\right\} = \ell - 1 + \frac{\ell + 2}{\nu +
      1}. 
  \end{equation*}
\end{lem}

\begin{proof}
  Let $\psi: \mathbb{Z}_{\geq 0} \rightarrow \mathbb{R}_{>0}$ be
  non-increasing.  Note first that
  \begin{align*}
    \mathcal{E}(\psi) &= \bigg\{(\alpha_1, \dots, \alpha_\ell) \in
    \mathbb{R}^{\ell} : \min_{1 \leq i \leq n - \ell} \norm{x_1
      \alpha_1 + \cdots + x_\ell
      \alpha_\ell + x_{\ell+i} \alpha_{\ell+i}} < \psi(H(\ux))\\
    & \quad \quad \text{for infinitely many } \ux \in \mathbb{Z}^n \\
    & \quad \quad \text{with } \max\{\abs{x_1}, \dots, \abs{x_\ell}\}
    >
    \max\{\abs{x_{\ell+1}}, \dots, \abs{x_n}\} \bigg\} \\
    &\subseteq \bigcup_{1 \leq i \leq n-\ell} \bigg\{(\alpha_1, \dots,
    \alpha_\ell) \in \mathbb{R}^{\ell} : \norm{x_1 \alpha_1 + \cdots +
      x_\ell
      \alpha_\ell + x_{\ell+i} \alpha_{\ell+i}} < \psi(H(\ux))\\
    & \quad\quad \text{ for infinitely many } \ux \in \mathbb{Z}^n,
    \text{ with } \max\{\abs{x_1}, \dots, \abs{x_{\ell}}\} >
    \abs{x_{\ell+i}}  \bigg\} \\
    & = \bigcup_{1 \leq i \leq n-\ell} \mathcal{E}_i(\psi),
  \end{align*}
  where $\mathcal{E}_i(\psi)$ is defined by the last equality.
  Furthermore, as the minimum in the definition of $\mathcal{E}$ can
  only be attained for finitely many values of $i$, there exists
  $i_0$ such that $1 \leq i_0 \leq n-\ell$ and
  $\mathcal{E}_{i_0}(\psi) \subseteq \mathcal{E}(\psi)$. The upshot is
  that
  \begin{equation*}
    \min_{1 \leq i \leq n-\ell} \dim (\mathcal{E}_i(\psi)) \leq \dim
    \mathcal{E}(\psi) \leq \max_{1 \leq i \leq n-\ell} \dim
    (\mathcal{E}_i(\psi)).
  \end{equation*}
  We calculate the dimension of a generic $\mathcal{E}_i$,
  say of $\mathcal{E}_1$.

  As in the proof of Theorem \ref{thm:lebesgue}, we restrict ourselves
  to the unit cube and consider the set $\mathcal{E}^* = \mathcal{E}_1
  \cap [0,1]^\ell$.  In analogy with the proof of Theorem
  \ref{thm:lebesgue}, let
  \begin{multline*}
    B(x_1, \dots, x_\ell, x_{\ell+1}) \\
    = \left\{(\alpha_1, \dots, \alpha_\ell) \in [0,1]^\ell :
      \norm{x_1\alpha_1 + \cdots + x_\ell \alpha_\ell + x_{\ell+1}
        \alpha_{\ell+1}} \leq \psi(H(\ux))\right\},
  \end{multline*}
  As in that proof, we find that
  \begin{equation}
    \label{eq:8}
    \abs{B(x_1, \dots, x_\ell, x_{\ell+1})} \asymp \psi(H(\ux)).
  \end{equation}
  Also, by the same argument as the one used in \cite{MR1258756}, if
  $(x_1, \dots, x_\ell)$ and $(x'_1, \dots, x'_\ell)$ are linearly
  independent, then for any $x_{\ell+1}, x'_{\ell+1}$,
  \begin{multline}
    \label{eq:10}
    \abs{B(x_1, \dots, x_\ell, x_{\ell+1}) \cap B(x'_1, \dots,
      x'_\ell, x'_{\ell+1})} \\
    =\abs{B(x_1, \dots, x_\ell, x_{\ell+1})}\cdot \abs{B(x'_1, \dots,
      x'_\ell, x'_{\ell+1})}.
  \end{multline}
  Finally, standard arguments from the proof of the one-dimensional
  Khintchine's Theorem (see \emph{e.g.} \cite{MR0087708}) show that if
  $(x_1, \dots, x_\ell)$ and $(x'_1, \dots, x'_\ell)$ are linearly
  dependent, and $(x_{\ell+1}, x'_{\ell+1}) = 1$, then
  \begin{multline}
    \label{eq:11}
    \abs{B(x_1, \dots, x_\ell, x_{\ell+1}) \cap B(x'_1, \dots,
      x'_\ell, x'_{\ell+1})} \\
    \ll \abs{B(x_1, \dots, x_\ell, x_{\ell+1})}\cdot \abs{B(x'_1, \dots,
      x'_\ell, x'_{\ell+1})}.
  \end{multline}
    
  Now, let
  \begin{align*}
    P_N = \big\{(x_1, \dots, x_\ell, x_{\ell+1}) \in
    \mathbb{Z}_{\geq0}^{\ell+1} \ : \ &H(x) =
    N, \\
    &\gcd(x_1, \dots, x_\ell, x_{\ell+1}) = 1 \text{ and } \\
    & x_{\ell+1} \text{ is prime with } x_{\ell+1} \leq N/2 \big\}.
  \end{align*}
  Let $\pi(x)$ denote the prime counting function, \emph{i.e.},
  \begin{equation*}
    \pi(x) = \left\{p \leq x : p \text{ is prime}\right\}.
  \end{equation*}
  Arguing again as in the proof of Theorem \ref{thm:lebesgue}, we find
  that
  \begin{equation*}
    \# P_N \asymp \sum_{d \vert N} \mu(d)
    \left(\frac{N}{d}\right)^{\ell-1} \pi(N/(2d)) \asymp \sum_{d \vert
      N} \mu(d) \left(\frac{N}{d}\right)^{\ell}
    \frac{1}{\log(N/(2d))}. 
  \end{equation*}
  The last asymptotic equality comes from the Prime Number Theorem. It
  is straightforward to check that if $\ux, \ux' \in \cup_{N \geq N_0}
  P_N$, then either \eqref{eq:10} or \eqref{eq:11} holds.

  If $\ell > 1$, as before by \eqref{eq:21}
  \begin{equation}
    \label{eq:9}
    \#P_N \gg \frac{N^\ell}{\log N} \sum_{d \vert N}  \mu(d)
    \left(\frac{1}{d}\right)^{\ell} \gg \frac{N^{\ell}}{\log N},
  \end{equation}
  and we find from usual arguments that
  \begin{equation*}
    \abs{\mathcal{E}^*} = 1
  \end{equation*}
  whenever
  \begin{equation}
    \label{eq:12}
    \sum_{h=1}^\infty \frac{h^\ell}{\log h} \psi(h) = \infty.
  \end{equation}
  When $\ell = 1$, the same conclusion is ensured by summing
  \eqref{eq:12} over dyadic blocks exactly as in the proof of Theorem
  \ref{thm:lebesgue}.

  On the other hand, it is a straigthforward consequence of \eqref{eq:8}
  and the Borel--Cantelli Lemma that
  \begin{equation*}
    \abs{\mathcal{E}^*} = 0
  \end{equation*}
  whenever
  \begin{equation}
    \label{eq:14}
    \sum_{h=1}^\infty h^\ell \psi(h) < \infty.
  \end{equation}

  As in the proof of Theorem \ref{thm:hausdorff}, we obtain an
  analogous Hausdorff measure result by invoking the slicing technique
  of \cite{MR2264714}. In the case $\ell = 1$, we use the
  one-dimensional version, known as the Mass Transference Principle
  from \cite{MR2259250}.

  Let $f: \mathbb{R} \rightarrow \mathbb{R}$ is a dimension function
  with $r \mapsto r^{-\ell} f(r)$ monotonically increasing and such
  that $g(r) = r^{-(\ell - 1)} f(r)$ is also a dimension function. We
  proceed to get upper and lower bounds on the Hausdorff $f$-measure
  of $\mathcal{E}^*$.

  The covering argument from the proof of Theorem \ref{thm:hausdorff}
  gives that
  \begin{equation*}
    \mathcal{H}^f(\mathcal{E}^*) = 0, 
  \end{equation*}
  whenever
  \begin{equation}
    \label{eq:15}
    \sum_{h=1}^\infty h^{\ell+1} g\left(\frac{\psi(h)}{h}\right) <
    \infty. 
  \end{equation}
  For the divergence case, an argument similar to that of the proof of
  Theorem \ref{thm:hausdorff} gives that 
  \begin{equation*}
    \mathcal{H}^f(\mathcal{E}^*(\psi)) = \infty, 
  \end{equation*}
  whenever
  \begin{equation}
    \label{eq:16}
    \sum_{h=1}^\infty h^{\ell+1} g\left(\frac{\psi(h)}{h\log h}\right) =
    \infty,
  \end{equation}
  where we have used the divergence condition \eqref{eq:12}.

  Now, consider the dimension function $f(r) = r^s$, where $s = \ell -
  1 + (l+2)/(\nu+1)$. We immediately see that for $\psi(h) =
  h^{-\nu}\log h$,
  \begin{equation*}
    \sum_{h=1}^\infty h^{\ell+1} g\left(\frac{\psi(h)}{h\log h}\right) =
    \infty = \sum_{h=1}^\infty \frac{1}{h} = \infty,
  \end{equation*}
  so that by \eqref{eq:16}, $\mathcal{H}^s(\mathcal{E}^*(\psi)) =
  \infty$. On the other hand, letting
  \begin{equation*}
    \psi_0(r) = h^{-\nu}(\log h)^{-2(\nu+1)/(\ell+2)},
  \end{equation*}
  we have,
  \begin{equation*}
    \sum_{h=1}^\infty h^{\ell+1} g\left(\frac{\psi(h)}{h}\right) =
    \sum_{h=1}^\infty \frac{1}{h(\log h)^2} < \infty,
  \end{equation*}
  so that by \eqref{eq:15}, $\mathcal{H}^s(\mathcal{E}^*(\psi_0)) =
  0$.  Plainly, the set we are estimating is a subset of
  $\mathcal{E}^*(\psi) \setminus \mathcal{E}^*(\psi_0)$, so it has the
  required dimension.
\end{proof}

Note that the proof of Lemma \ref{lem:auxiliary_exps} contains a
result somewhat weaker than the zero-one law of Theorem
\ref{thm:lebesgue}. Indeed, there is a gap between the series required
for the measure zero and the measure one case. We have no doubt that
this gap can be closed, and that the logarithmic factors in
\eqref{eq:12} and \ref{eq:16} can be removed. Nonetheless, we do not
consider the exponent defined here to be of enough interest on its own
to warrant a more detailed calculation. Furthermore, the present
result is sufficient for the purposes of this paper.

\begin{proof}[Proof of Theorem \ref{thm:small_exponents}]
  As in \cite{MR0429762}, we take $\alpha_{\ell+1}, \dots, \alpha_n
  \in \mathbb{R}$ with $1, \alpha_{\ell+1}, \dots, \alpha_n$ linearly
  independent over $\mathbb{Q}$ and such that for every $N$ large
  enough, there is an integer $q$ with $1 \leq q \leq N$ and
  \begin{equation}
    \label{eq:7}
    \norm{q \alpha_{\ell+i}} < e^{-N},
  \end{equation}
  with the possible exception of one value of $i$, say $i_0 = i_0(N)$.
  This is possible by Theorem 2 of \cite{MR0225728}.

  With $\alpha_{\ell + 1}, \dots, \alpha_n$ fixed, we take $\alpha_1,
  \dots, \alpha_\ell$ such that $1, \alpha_1, \dots, \alpha_n$ are
  linearly independent over $\mathbb{Q}$ and such that
  \begin{equation*}
    \tilde{\nu}_{n, \ell}(\alpha_{1}, \dots, \alpha_\ell) = \mu.
  \end{equation*}
  This is possible by Lemma \ref{lem:auxiliary_exps}.  Let $\epsilon$
  be a positive real number. Then,
  \begin{equation*}
    \norm{y_1 \alpha_1+ \cdots + y_{\ell} \alpha_\ell + y_{\ell+i}
      \alpha_{\ell+i}} > H(\uy)^{-\mu - \epsilon/3}
  \end{equation*}
  holds for any choice of integers $y_1, \dots, y_\ell, y_{\ell+i}$
  and any $i = 1, \dots, n-\ell$, if $\max\{\abs{y_1}, \allowbreak
  \dots, \abs{y_{\ell}}, \abs{y_{\ell+i}}\}$ is large enough.

  We show that
  \begin{equation*}
    \norm{x_1 \alpha_1 + \cdots + x_n\alpha_n} > H(\ux)^{-\mu - \epsilon}
  \end{equation*} 
  whenever $H(\ux)$ is large and $\ux$ is in the appropriate range.
  This immediately implies that $\mu_{n, \ell}(\alpha) \leq \mu +
  \epsilon$. Let $N = [\log H(\ux)]^2$ and choose an integer $q$ such
  that \eqref{eq:7} holds for all but one $i$. Suppose without loss of
  generality that $i_0(N) = 1$. Arguing in analogy with
  \cite{MR0429762}, recalling that $H(\ux)$ is attained among the
  first $\ell$ coordinates of $\ux$, we get
  \begin{align*}
    \norm{x_1 \alpha_1 + \cdots + x_n\alpha_n} & \geq q^{-1} \norm{x_1 q
      \alpha_1 + \cdots x_n q\alpha_n} \\
    & \geq q^{-1} \big(\norm{x_1 q \alpha_1 + \cdots + x_\ell q
      \alpha_{\ell} + x_{\ell+1} q
      \alpha_{\ell+1}} \\
    & \quad \quad \quad - H(\ux)\left(\norm{q
        \alpha_{\ell+2}} + \cdots + \norm{q \alpha_n}\right)\big) \\
    & > q^{-1} \left((q H(\ux))^{-\mu - \epsilon/3} -
      (n-\ell-1) H(\ux) e^{-N}\right)  \\
    & > q^{-1} \left(H(\ux)^{-\mu - (2\epsilon/3)} - (n-\ell -1) H(\ux)
      e^{-[\log H(\ux)]^2}\right) \\
    & > H(\ux)^{-\mu - \epsilon},
  \end{align*}
  when $H(\ux)$ is large enough.
  
  On the other hand, it is clear from the definition
  of the exponent $\tilde{\nu}_{n, \ell}$ that, for $\alpha$ chosen as above,
  \begin{equation*}
    \mu_{n,\ell}(\alpha) \geq \tilde{\nu}_{n, \ell}(\alpha_{1},
    \dots, \alpha_\ell) = \mu. 
  \end{equation*}
  Since $\epsilon$ is arbitrary, this gives the result.

  Since there are continuum many choices for $\alpha_{1}, \dots,
  \alpha_\ell$ by Lemma \ref{lem:auxiliary_exps}, we have completed
  the proof.  
\end{proof}

We conclude this section by assembling all the pieces required for a
proof of Theorem \ref{thm:main1}.

\begin{proof}[Proof of Theorem \ref{thm:main1}]
  The first part is an immediate consequence of Theorem
  \ref{thm:lebesgue}. The second part follows from Corollary
  \ref{cor:HDim} when $\mu_\ell \geq n$ and from Theorem
  \ref{thm:small_exponents} when $\mu_{\ell} \in [\ell+1, n)$.
\end{proof}

\section{On the difference between $\mu_{n, \ell}$ and $w_n$}
\label{sec:diff-betw-mu_n}

We now turn to the proof of Theorem \ref{thm:main2}. It depends on
earlier work by Bugeaud \cite{MR2007546}, which is related to
Schmidt's proof of the existence of $T$-numbers \cite{MR0279043,
  MR0344204}. In order to set the scene for the argument, we give some
background on these numbers. For additional details, the reader is
referred to \cite{MR2136100}.

In his 1932 classification of real numbers, Mahler
\cite{mahler32:_zur_approx_expon_logar} introduced for each positive integer
$n$ a Diophantine exponent for a real number $\xi$ by letting
\begin{multline*}
  w_n(\xi) = \sup\big\{w > 0 : 0 < \abs{P(\xi)} < H(P)^{-w}\\
  \text{ for infinitely many } P(X) \in \mathbb{Z}[X], \deg(P) \leq
  n\big\},
\end{multline*}
where $H(P)$ is the na\"ive height of the polynomial $P(X)$, \emph{i.e.},
the maximum of the absolute values among the coefficients of $P(X)$. 
Observe that $w_n (\xi)$ equals $w_n (\alpha)$, for the 
$n$-tuple $\alpha = (\xi^n, \ldots , \xi)$. 
A related quantity is
\begin{equation*}
  w(\xi) = \limsup_{n \rightarrow \infty} \frac{w_n(\xi)}{n}.
\end{equation*}

Using these quantities, Mahler classified the real numbers in four
classes.
\begin{itemize}
\item $\xi$ is an $A$-number if $w(\xi)=0$ (equivalently if $\xi$ is
  algebraic).
\item $\xi$ is an $S$-number if $w(\xi) < \infty$.
\item $\xi$ is a $T$-number if $w(\xi) = \infty$ but
  $w_n(\xi) < \infty$ for all $n$.
\item $\xi$ is a $U$-number if $w(\xi) = \infty$ and
  $w_n(\xi) = \infty$ for some $n$.
\end{itemize}
An elementary covering argument shows that almost all numbers are
$S$-numbers. Additionally, it is easy to see that Liouville numbers
such as $\sum 10^{-n!}$ are $U$-numbers. By contrast, it is very
difficult to prove that $T$-numbers exist. This was not accomplished
until 1970, when Schmidt showed how to construct examples
\cite{MR0279043, MR0344204} of such numbers.

In order to study the finer arithmetical properties of $T$-numbers,
and in particular to study the relation between Mahler's
classification and the related classification of Koksma
\cite{MR0000845}, Bugeaud \cite{MR2007546} refined Schmidt's
construction. In the process, the following result was obtained.

\begin{thm}[Theorem 3' of \cite{MR2007546}]
  \label{thm:bugeaud}
  Let $n \geq 3$ be an integer, let $\mu \in [0,1]$ and let $\nu >
  1$. Let $G(n) = 2n^3 + 2n^2+2n+1$ and let $\chi > G(n)$. Then there
  is a number $\lambda \in(0, 1/2)$, prime numbers $g_1, g_2, \dots$,
  with $g_1 \geq 11$ and integers $c_1, c_2 \dots$, such that for
  $\gamma_j = 2^{1/n}[g_j^{\mu}]$, the following conditions are satisfied:
  \begin{enumerate}
  \item[(I$_j$)] $g_j$ does not divide the norm of $c_j + \gamma_j$ for
    any $j$.
  \item[(II$_1$)] $\xi_1 = (c_1 + \gamma_1)/g_1 \in (1,2)$.
  \item[(II$_j$)] For any $j \geq 2$,
    \begin{equation*}
      \xi_j = \frac{c_j + \gamma_j}{g_j} \in \left(\xi_{j-1} -
        \tfrac{1}{2} g_{j-1}^{-\nu} , \xi_{j-1} +
        \tfrac{3}{4} g_{j-1}^{-\nu}\right) 
    \end{equation*}
  \item[(III$_1$)] For any algebraic number $\alpha \neq \xi_1$ of
    degree $\leq n$,
    \begin{equation*}
      \abs{\xi_1 - \alpha} \geq 2 \lambda H(\alpha)^{-\chi}.
    \end{equation*}
  \item[(III$_j$)] For any $j \geq 2$ and any algebraic number $\alpha
    \notin \{\xi_1, \dots, \xi_j\}$ of degree $\leq n$,
    \begin{equation*}
      \abs{\xi_j - \alpha} \geq \lambda H(\alpha)^{-\chi}.
    \end{equation*}
  \end{enumerate}
\end{thm}

It is a modification of this theorem, which will enable us to prove
Theorem \ref{thm:main2}. Rather than giving a complete proof (which
would be quite long), we choose to outline a few explanations, based
on Theorem \ref{thm:bugeaud}. We refer to the original paper
\cite{MR2007546} for the proof of this theorem.

\begin{proof}[Proof of Theorem \ref{thm:main2}]
  We will be working with a number and the powers of it. Hence, our
  goal consists in finding real numbers $\xi$ such that $\mu_{n, \ell}
  (\xi^n, \ldots , \xi)$ takes a prescribed (large) value for $\ell =
  1, \ldots , n$. We will use a construction analogous to the one of
  Theorem \ref{thm:bugeaud}.

  Let $n$ be an integer with $n \ge 2$.  Let $\gamma$ be a real
  algebraic number of degree $n$.  The
  general approach consists in contructing inductively a rapidly
  increasing sequence $(c_j)_{j \ge 1}$ of integers and a rapidly
  increasing sequence $(g_j)_{j \ge 1}$ of prime numbers such that,
  besides various technical conditions, the sequence $(\xi_j)_{j \ge
    1}$, where $\xi_j = (c_j + \gamma)/g_j$, is rapidly convergent to
  a real number $\xi$.  We do this in ensuring that the best algebraic
  approximants to $\xi$ of degree at most $n$ belong to the sequence
  $(\xi_j)_{j \ge 1}$ and, moreover, we control the differences $|\xi
  - \xi_j|$ in terms of the height $H(\xi_j)$ of $\xi_j$, that is, the
  maximal of the absolute values of the coefficients of its minimal
  polynomial. 

  More precisely, if $\lambda$ is a sufficiently large real number,
  the construction gives that $|\xi - \xi_j| \asymp g_j^{-\lambda}$
  and the height of $\xi_j$ is exactly known in terms of $g_j$.  In
  particular, if $\lambda_1$ and $\lambda_2$ are sufficiently large
  (for technical reasons) real numbers with $\lambda_1 < \lambda_2$,
  we are able to construct $\xi$ such that $|\xi - \xi_j| \asymp
  H(\xi_j)^{-\lambda_2}$ for any $j$ (see Condition $({\rm II}_{j+1})$ in
  Theorem \ref{thm:bugeaud} with $\nu = \lambda_2$), while $|\xi -
  \theta| \gg H(\theta)^{-\lambda_1}$ for any algebraic number
  $\theta$ of degree at most $n$ which is not in the sequence
  $(\xi_j)_{j \ge 1}$ (see Condition $({\rm III}_j)$ in
  \ref{thm:bugeaud} with $\chi = \lambda_1$).

  Actually, the construction of \cite{MR2007546} is flexible enough to
  give even more.  Take $\lambda, \lambda_1, \ldots, \lambda_n$ real
  numbers with $\lambda \le \lambda_1 \le \ldots \le \lambda_n$.  
  For $k= 1, \ldots , n$, we 
  are able indeed to construct $\xi$ such that $|\xi - \xi_j| \asymp
  H(\xi_j)^{-\lambda_k}$ for any $j$ congruent to $k$ modulo $n$,
  while $|\xi - \theta| \ge H(\theta)^{-\lambda}$ for any algebraic
  number $\theta$ of degree at most $n$ which is not in the sequence
  $(\xi_j)_{j \ge 1}$.

  Since we are here concerned with linear forms in $1, \xi, \ldots ,
  \xi^n$, we are not interested in the differences $|\xi - \theta|$,
  but merely in the values taken at $\xi$ by integer polynomials of
  degree at most $n$.  Denote by $P(X)$ the minimal defining
  polynomial of $\gamma$.  Then, provided that $g_j$ does not divide
  the norm of $c_j + \gamma$ (see Condition $({\rm I}_j)$ of Theorem
  \ref{thm:bugeaud}), the integer polynomial $Q_j (X) = P(g_j X -
  c_j)$ is the minimal defining polynomial of $\xi_j$. The
  construction allows us to control precisely the smallness of $|Q_j
  (\xi)|$, and to prove that $|Q(\xi)|$ is not too small when $Q(X)$
  is not a multiple of some polynomial $Q_j(X)$.

  Another important feature of this construction is that we do not
  have to use the same algebraic number $\gamma$ at each step $j$ of
  the process. Instead, we can work with a given sequence
  $(\gamma_j)_{j \ge 1}$ of real algebraic numbers of degree at most
  $n$.  Furthermore, it has been heavily used in \cite{MR2007546} that
  for $j \ge 1$ the algebraic number $\gamma_j$ may depend on $g_j$, 
  as in Theorem \ref{thm:bugeaud}.
  This remark introduces a flexibility that is crucial for the present
  proof.

  We now outline the difference between the proof of Theorem
  \ref{thm:bugeaud} found in \cite{MR2007546} and the present proof.

  Let $\ell = 1, \ldots , n$.  For $j \ge 1$, we select $P_j (X)$, the
  minimal polynomial of $\gamma_j$, in such a way that the height of
  the polynomial $P_j (g X - c) - P_j (- c)$ is equal to the
  coefficient of $X^{\ell}$, where $\ell$ is congruent to $j$ modulo
  $n$.  This means that on evaluating the polynomial $P_j (g X - c)$
  at $\xi$, we get a linear form in the powers of $\xi$, say $a_n
  \xi^n + \ldots + a_1 \xi + a_0$, where $|a_{\ell}| > \max\{|a_n|,
  \ldots , |a_{\ell - 1}|, |a_{\ell + 1}|, \ldots , |a_1|\}$.
  Choosing $\ell = 1$, this allows us to control precisely the small
  values of the linear form $||x_n \xi^n + \ldots + x_1 \xi||$ subject
  to the condition $|x_1| > \max\{|x_2|, \ldots , |x_n|\}$.  This
  corresponds exactly to the exponent $\mu_{n, 1} (\xi^n, \ldots ,
  \xi)$.  Similarly, we can control the exponents $\mu_{n, 2} (\xi^n,
  \ldots , \xi), \ldots , \mu_{n, n} (\xi^n, \ldots , \xi)$ by
  selecting $\ell$ appropriately. As we are controlling each exponent
  in a fixed residue class modulo $n$, we control simultaneously all
  exponents.

  We slightly modify the construction given in \cite{MR2007546}.
  Namely, we choose $c_j$ at each step in order to ensure
  \begin{equation}
    \label{eq:13}
    2^{2n+2} c_j \le g_j \le 2^{2n+3} c_j.
  \end{equation}
  With this choice, the resulting real number $\xi$ is lying in the
  interval $(2^{-2n-4}, 2^{-2n})$.

  Now, we give explicitly suitable minimal polynomials $P_j (X)$ of
  the numbers $\gamma_j$.  For $\ell = n$, that is, for $j$ divisible
  by $n$, we set $P_j(X) = X^n - 2 g_j^n$. Therefore, the minimal
  polynomial of $\xi_j$ is $(g_j X - c_j)^n - 2 g_j^n$, and, by
  \eqref{eq:13}, its largest coefficient is, besides the constant
  coefficient, equal to the coefficient of $X^n$. Note that we work
  here with the same polynomial as in the proof of Theorem
  \ref{thm:bugeaud} with the parameter $\mu = 1$.

  For any integer $\ell = 1, \ldots , n-1$ and any positive integer
  $a$, the polynomial $X^n - 2 (a X - 1)^{\ell}$ is irreducible, by
  Eisenstein's criterion applied with the prime $2$.  If $j$ is
  congruent to $\ell$ modulo $n$, we set
  \begin{equation}
    \label{eq:17}
    P_j(X) = X^n - 2 ([g_j^{(n-\ell)/\ell}] X - 1)^{\ell},
  \end{equation}
  where $[ \cdot ]$ denotes the integer part.  Therefore, the minimal
  polynomial of $\xi_j$ is
  \begin{equation*}
    Q_j (X) := P_j(g_j X - c_j) = (g_j X - c_j)^n - 2
    ([g_j^{(n-\ell)/\ell}] (g_j X - c_j) - 1)^{\ell},
  \end{equation*}
  and, by \eqref{eq:13}, its largest coefficient is equal to the
  coefficient of $X^{\ell}$.  Note that we work here with a family of
  polynomials of a similar shape to the one defined in Lemma 3 of
  \cite{MR2007546}. In particular, it is easily shown that $P_j(X)$ as
  in \eqref{eq:17} has exactly $\ell$ roots very close to 
  $1/[g_j^{(n-\ell)/\ell}]$ and that its other roots are not
  too close to each other.

  It remains for us to explain how one proceeds to control $|\xi -
  \xi_j|$.  Let $\mu_{n,1}, \ldots , \mu_{n,n}$ be real numbers with
  $\mu_{n, 1}$ sufficiently large and $\mu_{n,1} \le \ldots \le
  \mu_{n,n}$.  Instead of working with a single $\nu$ as in the proof
  of Theorem \ref{thm:bugeaud}, we work with a
  sequence $(\nu_j)_{j \ge 1}$. Observe that our choice for the
  polynomials $P_j(X)$ implies $H(\xi_j) \asymp g_j^n$ for $j \ge 1$.
  Let $\lambda_1, \ldots , \lambda_n$ be (large) real numbers to be
  chosen later on, and set $\nu_{j n + \ell} = \lambda_\ell$ for any
  $\ell = 1, \ldots , n$ and any $j \ge 1$.  Then $|\xi - \xi_j|
  \asymp H(\xi_j)^{- \lambda_{\ell}/n}$ if $j$ is congruent to $\ell$
  modulo $n$.

  We proceed as on page 101 of \cite{MR2007546}. Suppose that $j
  \equiv \ell \pmod{n}$. Then, $P_j(X)$ has exactly $\ell$ roots very
  close to each other with $\gamma_j$ being one of them. Let $\xi_j =
  \beta_{j1}, \beta_{j2}, \dots, \beta_{jn}$ denote the roots of
  $Q_j(X) = P_j(g_jX - c_j)$. We order these so that $\beta_{j1},
  \dots, \beta_{j\ell}$ correspond to the roots $\gamma_1, \dots,
  \gamma_\ell$ of $P_j(X)$ which are close.  It follows that 
  $\abs{\xi-\beta_{ji}} \asymp
  g_j^{-n^2 / \ell^2}$ for $i=2, \dots, \ell$.  Denote by
  $\gamma_{\ell+1}, \dots, \gamma_n$ the remaining roots of $P_j(X)$.
  Now, arguing as in \cite{MR2007546}, we get
  \begin{align*}
    |Q_j(\xi)| &= g_j^n \abs{\xi -\xi_j} \prod_{2 \leq i \leq \ell}
    \abs{\xi - \beta_{ji}} \prod_{\ell+1 \leq i \leq n} \abs{\xi -
      \beta_{ji}} \\
    &\asymp g_j^\ell H(\xi_j)^{- \lambda_{\ell}/n}
    g_j^{-(\ell-1)n^2 / \ell^2} \prod_{\ell+1 \leq i \leq n}
    \abs{\frac{1}{\left[g_j^{(n-\ell)/\ell}\right]} - \gamma_j}\\
    &\asymp H(\xi_j)^{- \lambda_{\ell}/n} g_j^{n - n^2 (\ell - 1)/\ell^2}
    \asymp H(\xi_j)^{- \delta_{\ell} - \lambda_{\ell}/n} \asymp
    H(Q_j)^{- \delta_{\ell} - \lambda_{\ell}/n}
  \end{align*}
  for $\delta_\ell := 1 - n (\ell - 1)/\ell$.  Here, we have used
  Lemma 6 of \cite{MR2007546} in order to control the product over the
  last $n-\ell$ roots. Note that for $\ell > 1$, $\delta_\ell$ is a
  negative number. However, this is of no importance for the
  approximation properties studied here, since we still have freedom
  to choose the $\lambda_\ell$.

  The fact that $\eta_{\ell}$ depends only on $\ell$ is a consequence
  of the particular shape of the polynomials $Q_j(X)$.  It is now
  sufficient to select $\lambda_{\ell}$ in such a way that
  $\delta_{\ell} + \lambda_{\ell}/n = \mu_{n, \ell}$.  With this
  choice, we get
  \begin{equation}
    \label{eq:20}
    \mu_{n, \ell} (\xi^n, \ldots , \xi) \geq \mu_{n, \ell},
    \quad \ell \geq 1, \ldots , n,
  \end{equation}
  as expected.

  The final estimate needed is a lower bound for $\abs{Q(\xi)}$ when
  $Q(X) \neq Q_j(X)$ for any $j$.  In order to obtain such a bound, we
  argue again as in \cite{MR2007546}.  Let $Q(X) = a R_1(X)\cdots
  R_p(X)$ be a factorisation of $Q(X) \neq Q_j(X)$, a polynomial of
  degree at most $n$, into primitive irreducibles. Using the property
  analogous to (III$_j$) of Theorem \ref{thm:bugeaud} with the present
  polynomials, we find in analogy with equation (24) of
  \cite{MR2007546}, that for $1 \leq i \leq p$,
  \begin{align*}
    \abs{R_i(\xi)} &\gg H(R_i)^{2-\deg(R_i)}\abs{\xi - \alpha} \gg
    H(R_i)^{-\lambda -\deg(R_i) +2}\\
    & \gg H(R_i)^{-\lambda-n+2} \gg H(R_i)^{-\mu_{n,1}}.
  \end{align*}
  The last inequality follows on insisting that $\mu_{n,1}$ is large
  enough.  Using the so-called Gelfond-inequality,
  \begin{align*}
    \abs{Q(\xi)} \gg (H(R_1) \cdots H(R_p))^{-\mu_{n,1}} \gg
    H(Q)^{-\mu_{n,1}}. 
  \end{align*} 
  It immediately follows that every polynomial taking $(\xi^n, \dots,
  \xi)$ close sufficiently close to zero is found among the $Q_j$, so
  that
  \begin{equation*}
    \mu_{n, \ell} (\xi^n, \ldots , \xi) \leq \mu_{n, \ell},
    \quad \ell = 1, \ldots , n.
  \end{equation*}
  Together with \eqref{eq:20}, this completes the proof that the tuple
  $\alpha = (\xi^n, \xi^{n-1}, \dots, \xi)$ satisfies all the desired
  equalities. Additionally, there is still enough flexibility in the
  construction to ensure that there are continuum many such $\xi$.
  This completes the proof.
  
  To conclude, we point out that we can construct a suitable $\alpha$
  with $\mu_{n, 1} (\alpha) \ll n^3$, which ensures that the exponents
  are not all infinite. Our process is, like in \cite{MR2007546},
  effective.
\end{proof}

\section{Lower bounds for the exponents $\mu_{n, \ell}$}
\label{sec:lower-bounds-expon}

Using the exponents $w_2$ and $\hat{w}_2$, it is easily seen that
Lemma 1 of Schmidt \cite{MR2264714} can be rewritten as
\begin{equation*}
  \mu_{2,1} \ge \frac{\hat{w}_2}{\hat{w}_2 - 1}.
\end{equation*}

Its proof can be straightforwardly extended to arbitrary $n$ and $\ell$.
This was already done by Thurnheer for $\ell = n-1$.

\begin{prop}
  \label{prop:thurnheer}
  Let $n \ge 2$ be an integer.
  For any $n$-tuple $\alpha$ 
  and any integer $\ell = 1, \ldots , n$, we have
  \begin{equation}
    \label{eq:19}
    \mu_{n, \ell} (\alpha) \ge  \frac{\ell \, \hat{w}_n
      (\alpha)}{\hat{w}_n (\alpha)  - n + \ell}.
  \end{equation}
\end{prop}

\begin{proof}
For simplicity, we write $\mu_{n, \ell}$ for $\mu_{n, \ell} (\alpha)$
and $\hat{w}_n$ for $\hat{w}_n (\alpha)$.
Without loss of generality, we may asume that $\ell < \mu_{n, \ell} < n$.
Let $\eta \ge 1$ be a real number.
Consider the convex body $\mathcal{B}$ given by the equations
\begin{align*}
  |x_1|, \ldots , |x_{\ell}| &\le N^{\eta}, \\
  |x_{\ell + 1}|, \ldots , |x_n| &\le N,\\
  |x_1 \alpha_1 + \ldots + x_n \alpha_n + x_0| &\le N^{-\ell \eta - n
    + \ell} = (N^{\eta})^{-(\ell \eta + n - \ell)/\eta}. 
\end{align*}
By Minkowski's theorem, it contains a non-zero point with integer
coordinates.  Let $\epsilon$ be a positive real number with $\epsilon
< \mu_{n, \ell} - \ell$ and $\epsilon < n - \mu_{n, \ell}$.  The
definition of the exponent $\mu_{n, \ell}$ implies that, when $N$ is
sufficiently large, the system of equations
\begin{align*}
  |x_1|, \ldots , |x_n| &\le N^{\eta}, \\
  \max\{|x_1|, \ldots , |x_\ell|\} &> \max\{|x_{\ell + 1}|, \ldots ,
  |x_n|\}, \\
  |x_1 \alpha_1 + \ldots + x_n \alpha_n + x_0| &\le
  (N^{\eta})^{-(\mu_{n, \ell} + \epsilon)},
\end{align*}
has no solution.  Consequently, if $\eta$ is defined by
\begin{equation*}
  \mu_{n, \ell} + \epsilon = \frac{\ell \eta + n - \ell}{\eta},
\end{equation*}
then, for large $N$, any non-zero integer point $(x_0, x_1, \ldots ,
x_n)$ in $\mathcal{B}$ satisfies
\begin{equation*}
  \max\{|x_1|, \ldots , |x_\ell|\}
  \le \max\{|x_{\ell + 1}|, \ldots , |x_n|\} \le N.
\end{equation*}
This shows in turn that
\begin{equation*}
  \hat{w}_n \ge \ell \eta + n - \ell =  (n- \ell)
  \left(1 + \frac{\ell}{\mu_{n, \ell} + \epsilon - \ell}\right),
\end{equation*}
which gives the desired inequality when $\epsilon$ tends to zero.
\end{proof}

Since $w_n$ is almost always equal to $n$ (see \cite{MR0245527}) and
$\mu_{n, \ell} \le w_n$, it immediately follows from Proposition
\ref{prop:thurnheer} that the exponent $\mu_{n, \ell}$ is almost
always equal to $n$. This gives an alternative proof of the first
assertion of Theorem \ref{thm:main1}.

When we follow the second part of the proof of Theorem 1 of Schmidt
\cite{MR0429762}, we see that he 
actually established the inequality
\begin{equation*}
  \mu_{2, 1} \ge \hat{w}_2 - 1 + \frac{\hat{w}_2}{w_2},
\end{equation*}
although he only used the (often) weaker inequality
$\mu_{2, 1} \ge \hat{w}_2 - 1$.

Likewise, Thurnheer \cite{MR1056107} extended in 1990 Schmidt's result by
proving that $\mu_{n, n-1} \allowbreak \ge \hat{w}_n - 1$, but his paper
contains the proof of the lower bound
\begin{equation*}
  \mu_{n, n-1} \ge \hat{w}_n - 1 + \frac{\hat{w}_n}{w_n},
\end{equation*}
as given in Theorem \ref{thm:main3}.

\section{Concluding remarks}
\label{sec:restr-appr-along}

It is interesting to note that while the results proved by metrical
methods, \emph{i.e.}, Theorem \ref{thm:main1}, are results in all of
$\mathbb{R}^n$, the explicit constructions of Theorem \ref{thm:main2}
are carried out on the Veronese curve $(\xi^n, \xi^{n-1}, \dots, \xi)
\subseteq \mathbb{R}^n$. For the classical exponent $w_n = \mu_{n,n}$,
the metrical theory for Lebesgue measure is the same in the two
settings, as shown by Beresnevich \cite{MR1709049} for the Veronese
curves and more generally for non-degenerate manifolds by Beresnevich,
Bernik, Kleinbock and Margulis \cite{MR1944505}.

We have not been able to show that the metrical theory remains the
same when restricted to non-degenerate curves and manifolds for 
$\mu_{n, \ell}$ with $\ell
< n$. Nonetheless, it remains of interest whether the results of the
present paper may be extended or improved on such sets. 

For what it is worth, if $\alpha = (\xi^n, \xi^{n-1}, \dots, \xi)$, then
$\hat{w}_n(\alpha)$ is at most $2n-1$, as established by
Davenport and Schmidt \cite{MR0246822}. Hence, using Theorem
\ref{thm:main3},
\begin{equation}
  \label{eq:18}
  \mu_{n, \ell} (\alpha) \ge 2 \ell - \frac{\ell (2\ell - 1)}{n-1 + \ell}.
\end{equation}
Inserting $\ell = 1$ and letting $n$ increase, this provides a
positive answer to Problem 2 along Veronese curves. 

\providecommand{\bysame}{\leavevmode\hbox to3em{\hrulefill}\thinspace}
\providecommand{\MR}{\relax\ifhmode\unskip\space\fi MR }
\providecommand{\MRhref}[2]{%
  \href{http://www.ams.org/mathscinet-getitem?mr=#1}{#2}
}
\providecommand{\href}[2]{#2}

\end{document}